\documentclass[12pt, oneside]{amsart}
\usepackage{latexsym,amssymb,amsfonts,amsmath, amscd}
\usepackage[T1]{fontenc}
\usepackage[pdftex]{graphicx}
\usepackage{pgf,tikz}
\usetikzlibrary{arrows, arrows.meta}
\usetikzlibrary{positioning,calc}
\usepackage{amssymb, amsmath, amsthm, enumerate}
\usepackage{url}
\usepackage[top=3cm, bottom=3cm, left=3cm, right=2cm]{geometry}
\usepackage[noadjust]{cite}

\newtheorem{theorem}{Theorem}[section]

\newtheorem{corollary}[theorem]{Corollary}
\newtheorem{proposition}[theorem]{Proposition}

\newtheorem{lemma}[theorem]{Lemma}

\theoremstyle{definition}
\newtheorem{definition}[theorem]{Definition}

\theoremstyle{remark}
\newtheorem{remark}[theorem]{Remark}
\newtheorem{example}[theorem]{Example}

\newcommand{\dist}{\operatorname{dist}}
\newcommand{\diam}{\operatorname{diam}}
\newcommand{\acr}{\newline\indent}

\newcommand{\RR}{\mathbb{R}}
\newcommand{\NN}{\mathbb{N}_0}
\newcommand{\CC}{\mathbb{C}}
\newcommand{\bB}{\overline{\mathbf{B}}}
\let\ol\overline
\newcommand{\myeq}{\stackrel{r}{=}}

\numberwithin{equation}{section}

\makeatletter

\renewcommand{\p@enumii}{}

\makeatother

\begin{document}

\title{Bipartite graphs and best proximity pairs}

\author{Karim Chaira, Oleksiy Dovgoshey, and Samih Lazaiz}

\address{Karim Chaira\acr
Ben M'sik Faculty of Sciences\acr
Hassan II University, Casablanca, Morocco}
\email{chaira\_karim@yahoo.fr}

\address{Oleksiy Dovgoshey\acr
Institute of Applied Mathematics and Mechanics of NASU\acr
Dobrovolskogo str. 1, Slovyansk 84100, Ukraine}
\email{oleksiy.dovgoshey@gmail.com}

\address{Samih Lazaiz\acr
ENSAM Casablanca, Hassan II University\acr
Casablanca, Morocco}
\email{samih.lazaiz@gmail.com}



\begin{abstract}
We say that a bipartite graph \(G(A, B)\) with fixed parts \(A\), \(B\) is proximinal if there is a semimetric space \((X, d)\) such that \(A\) and \(B\) are disjoint proximinal subsets of \(X\) and all edges \(\{a, b\}\) satisfy the equality \(d(a, b) = \dist(A, B)\). It is proved that a bipartite graph \(G\) is not isomorphic to any proximinal graph iff \(G\) is finite and empty. It is also shown that the subgraph induced by all non-isolated vertices of a nonempty bipartite graph \(G\) is a disjoint union of complete bipartite graphs iff \(G\) is isomorphic to a nonempty proximinal graph for an ultrametric space.
\end{abstract}

\maketitle

\section{Introduction and preliminaries}

Let \(X\) be a set. A \emph{semimetric} on \(X\) is a function \(d\colon X \times X \to [0, \infty)\) such that \(d(x, y) = d(y, x)\) and \((d(x, y) = 0) \Leftrightarrow (x = y)\) for all \(x\), \(y \in X\). A pair \((X, d)\), where \(d\) is a semimetric on \(X\), is called a \emph{semimetric space} (see, for example, \cite[p.~7]{Blumenthal1953}). A semimetric \(d\) is a \emph{metric} if the \emph{triangle inequality} 
\[
d(x, y) \leqslant d(x, z) + d(z, y)
\]
holds for all \(x\), \(y\), \(z \in X\). A metric is an \emph{ultrametric} if we have the \emph{strong triangle inequality} 
\[
d(x, y) \leqslant \max\{d(x, z), d(z, y)\}
\]
instead of the triangle one. We shall denote by \(\mathbf{SM}\), \(\mathbf{M}\) and \(\mathbf{UM}\) the classes of all nonvoid semimetric spaces, nonvoid metric spaces and, respectively, ultrametric ones. In what follows, we will write \(D(X, d)\) for the \emph{distance set} of \((X, d) \in \mathbf{SM}\),
\[
D(X, d) := \{d(x, y) \colon x, y \in X\}.
\]
and denote by \(\diam (A)\) the diameter of set \(A \subseteq X\),
\[
\diam (A) := \sup\{d(x, y) \colon x, y \in A\}.
\]

Let \((X, d)\) be a semimetric space. A \emph{closed ball} with \emph{radius} \(r \geqslant 0\) and a \emph{center} \(c \in X\) is the set 
\[
\overline{B}_r(c) := \{x \in X \colon d(c, x) \leqslant r\}. 
\]
We will denote by \(\bB_{X}\) the set of all closed balls in \((X, d)\).

Let \(A\) and \(B\) be two nonempty subsets of a semimetric space \(X\). If \(A\) and \(B\) are disjoint, \(A \cap B = \varnothing\), it is possible to consider the problem of finding a point \(x\in A\) which is as close as possible to \(B\). Problems of this type are typical for Approximation Theory. 

The following is a ``semimetric'' modification of the corresponding definition from \cite{ref12}.

\begin{definition}\label{d1.1}
Let \((X, d) \in \mathbf{SM}\). A set \(A \subset X\) is said to be \emph{proximinal} in \((X, d)\) if, for every \(x\in X\), there exists \(a_0\in A\) such that 
\begin{equation}\label{d1.1:e1}
d(x,a_0) = \inf\{d(x, a)\colon a\in A\}.
\end{equation}
The point \(a_0\), if it exists, is called a \emph{best approximation} to \(x\) in \(A\).
\end{definition}

\begin{remark}
Since every \((X, d) \in \mathbf{SM}\) is nonempty, Definition~\ref{d1.1} implies that all proximinal sets also are nonempty.
\end{remark}

Let \(A\) and \(B\) be subsets of a space \((X, d) \in \mathbf{SM}\). We will say that the pair \((A, B)\) is \emph{proximinal} if \(A\) and \(B\) are proximinal in \((X, d)\).

Some results connected with existence of the best approximations in metric spaces can be found in \cite{ref9,ref10,ref11, Sch1985}. 

For nonempty subsets \(A\) and \(B\) of a semimetric space \((X,d)\), we define a distance from \(A\) to \(B\) as
\begin{equation}\label{e1.2}
\dist(A,B) := \inf\{d(a,b)\colon a\in A\ \text{and}\ b\in B\}.
\end{equation}
If \(A\) is a one-point set, \(A = \{a\}\), then, for brevity, we write \(\dist(a, B)\) instead of \(\dist(\{a\}, B)\).

It should be noted here that if \(A\) and \(B\) are infinite proximinal subsets in a semimetric space \((X, d)\), then, in general, there is no reason to have \(d(a, b) = \dist(A, B)\) for some \(a \in A\) and \(b \in B\). This led to the notion of the best proximity pairs. For example, in the proof of Theorem~\ref{t2.1} \(A\) and \(B\) are infinite proximinal subsets of an ultrametric space \((X, d_1)\), but \(d(a, b) > \dist(A, B)\) for all \(a \in A\) and \(b \in B\).

\begin{definition}\label{d1.2}
Let \((X, d) \in \mathbf{SM}\), and let \(A\) and \(B\) be nonempty subsets of \(X\). Write
\begin{gather}\label{d1.2:e1}
A_0 :=\{x\in A\colon d(x, y) = \dist(A, B)\ \text{for some}\ y\in B\},\\
\label{d1.2:e2}
B_0 := \{y\in B\colon d(x, y) = \dist(A, B)\ \text{for some}\ x\in A\}.
\end{gather}
A pair \((a_0, b_0) \in A_0 \times B_0\) for which \(d(a_0, b_0) = \dist(A, B)\) is called a \emph{best proximity pair} for the sets \(A\) and \(B\).
\end{definition}

For the case when \((X, d) \in \mathbf{M}\), Definition~\ref{d1.2} becomes Definition~1.1 from~\cite{ref12}.

\begin{remark}\label{r1.4}
It is easy to see that the following conditions are equivalent for all nonempty subsets \(A\), \(B\) of each \((X, d) \in \mathbf{SM}\):
\begin{itemize}
\item \(A_0 \neq \varnothing\);
\item \(B_0 \neq \varnothing\);
\item There is a best proximity pair for \(A\) and \(B\).
\end{itemize}
If \((a_0, b_0) \in A_0 \times B_0\) is a best proximity pair for \(A\) and \(B\), then, using the inclusions \(A_0 \subseteq A\), \(B_0 \subseteq B\) and formula~\eqref{e1.2}, we obtain 
\[
d(a_0, b_0) \geqslant \dist (A_0, B_0) \geqslant \dist (A, B) = d(a_0, b_0).
\]
Thus, the equality \(\dist (A_0, B_0) = \dist (A, B)\) holds and a pair \((x, y) \in A_0 \times B_0\) is a best proximity pair for \(A\) and \(B\) if and only if \((x, y)\) is a best proximity pair for \(A_0\) and \(B_0\).
\end{remark}

The next basic for us concept is the notion of graph.

A \emph{simple graph} is a pair \((V, E)\) consisting of a nonempty set \(V\) and a set \(E\) whose elements are unordered pairs of different elements of \(V\). In what follows, we will consider the simple graphs only.

For a graph \(G = (V, E)\), the sets \(V = V (G)\) and \(E = E(G)\) are called the \emph{set of vertices} and the \emph{set of edges}, respectively. Two vertices \(u\), \(v \in V\) are \emph{adjacent} if \(\{u, v\}\) is an edge in \(G\). A vertex \(v \in V(G)\) is \emph{isolated} if there are no vertices which are adjacent with \(v\) in \(G\). We say that \(G\) is \emph{empty} if \(E(G) = \varnothing\). Thus, \(G\) is empty iff all vertices of \(G\) are isolated. 

A graph \(H\) is, by definition, a \emph{subgraph} of a graph \(G\) if the inclusions \(V (H) \subseteq V (G) \) and \(E(H) \subseteq E(G)\) are valid. 

If \(G\) is a nonempty graph, then we will denote by \(G'\) a subgraph of \(G\) whose vertices are non-isolated vertices of \(G\) and such that \(E(G') = E(G)\).

\begin{remark}\label{r1.5}
The graph \(G'\) can be characterized by the following extremal property: If \(G' \subseteq H \subseteq G\) holds and \(H\) does not have any isolated vertices, then \(G' = H\). It is easy to see that \(V(G')\) is the union of all two-point sets \(\{a, b\} \in E(G)\).
\end{remark}

A graph \(G\) is \emph{finite} if \(V (G)\) is a finite set, \(|V (G)| < \infty\). A \emph{path} is a finite nonempty graph \(P\) whose vertices can be numbered so that
\[
V (P) = \{x_0, x_1, \ldots, x_k\}, \quad k \geqslant 1, \quad \text{and} \quad E(P) = \{\{x_0, x_1\}, \ldots, \{x_{k-1}, x_k\}\}.
\]
In this case we say that \(P\) is a path joining \(x_0\) and \(x_k\). A graph \(G\) is \emph{connected} if, for every two distinct \(u\), \(v \in V (G)\), there is a path \(P \subseteq G\) joining \(u\) and~\(v\).

Let \(\mathcal{F}\) be a set of graphs such that \(|\mathcal{F}| \geqslant 2\) and \(V(G_1) \cap V(G_2) = \varnothing\) for all distinct \(G_1\), \(G_2 \in \mathcal{F}\). A graph \(H\) is called the \emph{disjoint union} of graphs \(G \in \mathcal{F}\) if
\[
V(H) = \bigcup_{G \in \mathcal{F}} V(G) \quad \text{and} \quad E(H) = \bigcup_{G \in \mathcal{F}} E(G).
\]
It is easy to prove that a graph is connected iff it is not a disjoint union of some graphs.

\begin{definition}\label{d1.4}
A graph \(G\) is \emph{bipartite} if the vertex set \(V(G)\) can be partitioned into two nonvoid disjoint sets, or \emph{parts}, in such a way that no edge has both ends in the same part. A bipartite graph in which every two vertices from different parts are adjacent is called \emph{complete bipartite}.
\end{definition}

Below we will consider the bipartite graphs having the vertex sets of arbitrary cardinality.

Let us introduce now a notion of \emph{proximinal graph}.

\begin{definition}\label{d1.5}
A bipartite graph \(G = G(A, B)\) with fixed parts \(A\) and \(B\) is proximinal if there exists \((X, d) \in \mathbf{SM}\) such that \(A\) and \(B\) are disjoint proximinal subsets of \(X\), and the equivalence
\begin{equation}\label{d1.5:e1}
\bigl(\{a, b\} \in E(G)\bigr) \Leftrightarrow \bigl(d(a, b) = \dist(A, B)\bigr)
\end{equation}
is valid for every \(a \in A\) and every \(b \in B\). In this case we write
\[
G = G_X(A, B) = G_{X, d}(A, B)
\]
and say that \(G\) is proximinal for \((X, d)\).
\end{definition}

\begin{remark}\label{r1.7}
If \(G\) is a nonempty proximinal graph with parts \(A\) and \(B\), then it follows directly from the definitions that the equality \(V(G') = A_0 \cup B_0\) holds, where \(A_0\) and \(B_0\) are defined by~\eqref{d1.2:e1} and \eqref{d1.2:e2}, respectively. Moreover, vertices \(a\), \(b \in V(G)\) are adjacent iff \((a, b)\) is a best proximity pair for \(A\) and \(B\).
\end{remark}

Let \(G = G_{X, d}(A, B)\) be a proximinal graph for a semimetric space \((X, d)\). If we define \(\rho \colon X \times X \to [0, \infty)\) as
\[
\rho(x, y) = \begin{cases}
0 & \text{if } x = y,\\
c + d(x, y) & \text{if } x \neq y,
\end{cases}
\]
where \(c > 0\) is arbitrary, then \(\rho\) is a new semimetric on \(X\), \(A\) and \(B\) are disjoint proximinal sets in \((X, \rho)\), and \(G\) is proximinal for \((X, \rho)\), and \(G_{X, d}(A, B) = G_{X, \rho}(A, B)\). Thus, if in Definition~\ref{d1.5} we replace the condition \(A \cap B = \varnothing\) to the more strong condition 
\[
\dist(A, B) > 0,
\]
then the new definition will be equivalent to the original one.

Now we recall the concept of isomorphic graphs.

\begin{definition}\label{d1.6}
Let \(G_1\) and \(G_2\) be graphs. A bijection \(f \colon V(G_1) \to V(G_2)\) is an \emph{isomorphism} of \(G_1\) and \(G_2\) if
\[
(\{u, v\} \in E(G_1)) \Leftrightarrow (\{f(u), f(v)\} \in E(G_2))
\]
is valid for all \(u\), \(v \in V(G_1)\). The graphs \(G_1\) and \(G_2\) are \emph{isomorphic} if there exists an isomorphism of \(G_1\) and \(G_2\). 
\end{definition}

Here we use a classical example to illustrate the concept of proximinal graphs.

\begin{figure}[ht]
\begin{tikzpicture}[scale=1,
arrow/.style = {-{Stealth[length=5pt]}, shorten >=2pt}]
\def\xx{1.8cm}
\def\yy{0.9cm}
\coordinate [label=above:{$(1, 0, 0)$}] (A1) at (0*\xx, \yy);
\coordinate [label=above:{$(0, 1, 0)$}] (A2) at (1*\xx, \yy);
\coordinate [label=above:{$(0, 0, 1)$}] (A3) at (2*\xx, \yy);
\coordinate [label=above:{$(1, 1, 1)$}] (A4) at (3*\xx, \yy);

\coordinate [label=below:{$(1, 1, 0)$}] (B1) at (0*\xx, -\yy);
\coordinate [label=below:{$(1, 0, 1)$}] (B2) at (1*\xx, -\yy);
\coordinate [label=below:{$(0, 1, 1)$}] (B3) at (2*\xx, -\yy);
\coordinate [label=below:{$(0, 0, 0)$}] (B4) at (3*\xx, -\yy);

\draw [fill, black] (A1) circle (2pt);
\draw [fill, black] (A2) circle (2pt);
\draw [fill, black] (A3) circle (2pt);
\draw [fill, black] (A4) circle (2pt);
\draw [fill, black] (B1) circle (2pt);
\draw [fill, black] (B2) circle (2pt);
\draw [fill, black] (B3) circle (2pt);
\draw [fill, black] (B4) circle (2pt);

\draw (A1) -- (B1) -- (A2) -- (B3) -- (A3) -- (B4) -- (A1) -- (B2) -- (A3);
\draw (A2) -- (B4);
\draw (B1) -- (A4) -- (B2);
\draw (A4) -- (B3);
\draw (0, 1cm +\yy) node {\(G_{X, d}(A, B)\)};

\def\xx{1.7cm}
\begin{scope}[xshift=8.5cm, yshift=-\yy]
\coordinate [label=above:{$(1, 0, 0)$}] (a1) at (0*\xx, 2*\yy);
\coordinate [label=above:{$(1, 0, 1)$}] (a2) at (3*\xx, 2*\yy);
\coordinate [label=below:{$(0, 0, 1)$}] (a3) at (3*\xx, -2*\yy);
\coordinate [label=below:{$(0, 0, 0)$}] (a4) at (0*\xx, -2*\yy);

\coordinate [label=above:{$(1, 1, 0)$}] (b1) at (1*\xx, \yy);
\coordinate [label=above:{$(1, 1, 1)$}] (b2) at (2*\xx, \yy);
\coordinate [label=below:{$(0, 1, 1)$}] (b3) at (2*\xx, -\yy);
\coordinate [label=below:{$(0, 1, 0)$}] (b4) at (1*\xx, -\yy);

\draw [fill, black] (a1) circle (2pt);
\draw [fill, black] (a2) circle (2pt);
\draw [fill, black] (a3) circle (2pt);
\draw [fill, black] (a4) circle (2pt);
\draw [fill, black] (b1) circle (2pt);
\draw [fill, black] (b2) circle (2pt);
\draw [fill, black] (b3) circle (2pt);
\draw [fill, black] (b4) circle (2pt);

\draw (a1) -- (a2) -- (a3) -- (a4) -- (a1) -- (b1);
\draw (a2) -- (b2) -- (b3) -- (b4) -- (b1) -- (b2);
\draw (a3) -- (b3);
\draw (a4) -- (b4);
\draw (0, 1cm + 2*\yy) node {\(Q_3\)};
\end{scope}
\end{tikzpicture}
\caption{}\label{fig1}
\end{figure}

\begin{example}\label{ex1.10}
Let \(X\) be the set of all sequences \(\widetilde{q} = (q_1, q_2, q_3)\), where each \(q_i \in \{0, 1\}\). Let us denote by \(d(\widetilde{p}, \widetilde{q})\) the Hamming distance between \(\widetilde{p}\), \(\widetilde{q} \in X\),
\[
d(\widetilde{p}, \widetilde{q}) = \sum_{i=1}^3 |p_i - q_i|.
\]
Then \((X, d)\) is a metric space, and the sets
\[
A = \{(1, 0, 0), (0, 1, 0), (0, 0, 1), (1, 1, 1)\}, \quad 
B = \{(1, 1, 0), (1, 0, 1), (0, 1, 1), (0, 0, 0)\} 
\]
are disjoint proximinal subsets of \((X, d)\). Since the equality \(\dist(A, B) = 1\) holds, two sequences \(\widetilde{p} = (p_1, p_2, p_3)\) and \(\widetilde{q} = (q_1, q_2, q_3)\) are adjacent in \(G_{X, d}(A, B)\) iff they are different in exactly one place. The proximinal graph \(G_{X, d}(A, B)\) coincides, up to isomorphism, with the graph of the cube \(Q_3\) (see Figure~\ref{fig1}).
\end{example}

The goal of the paper is to characterize the proximinal graphs for semimetric, metric and ultrametric spaces. Theorem~\ref{t2.1} describes the structure of bipartite graphs which are proximinal for semimetric and metric spaces. Corollary~\ref{c2.2} of this theorem shows that a bipartite graph \(G\) is not isomorphic to any proximinal graph iff \(G\) is finite and empty. The structure of bipartite graphs, which are proximinal for ultrametric spaces, is completely described in Theorem~\ref{t2.9}. Corollary~\ref{c2.13} characterizes up to isomorphism the proximinal graphs for ultrametric spaces via disjoint unions of complete bipartite graphs.

In the last section of the paper, we introduce the farthest graphs \(G\) as bipartite graphs with fixed parts \(A\), \(B\) for which there are semimetric spaces \((X, d)\) such that \(A\) and \(B\) are disjoint subsets of \(X\), and two points \(a \in A\) and \(b \in B\) are adjacent in \(G\) iff they are maximally distant from each other. The structure of farthest graphs is described in Theorem~\ref{t3.2}. In Proposition~\ref{p3.4} it is shown that the farthest graphs and the proximinal graphs are the same up to isomorphism.

A special kind of bipartite graphs, the trees, gives a natural language for description of ultrametric spaces \cite{Carlsson2010, DLW, Fie, GV2012DAM, HolAMM2001, H04, BH2, Lemin2003, Bestvina2002, DDP2011pNUAA, DP2019PNUAA, DPT2017FPTA, DPT2015, Pet2018pNUAA, DP2018pNUAA, Dov2020TaAoG, BS2017, DP2020pNUAA, DKa2021, Dov2019pNUAA, DP2013SM, PD2014JMS}, but the authors are aware of only papers \cite{BDK2021a} and \cite{PD2014JMS}, in which complete bipartite and, more generally, complete multipartite graphs are systematically used to study ultrametric spaces. We note also that \cite{SLA2020IJoMaMS} and \cite{SV2017AGT} contain some results describing the behavior of nonexpansive mappings and best proximity pairs in the language of directed graphs.

\section{Characterization of proximinal graphs}

We start by characterizing the proximinal graphs corresponding to the general semimetric and metric spaces.

\begin{theorem}\label{t2.1}
Let \(G\) be a bipartite graph with fixed parts \(A\) and \(B\). Then the following statements are equivalent:
\begin{enumerate}
\item\label{t2.1:s1} Either \(G\) is nonempty or \(G\) is empty but \(A\) and \(B\) are infinite.
\item\label{t2.1:s2} \(G\) is proximinal for a metric space.
\item\label{t2.1:s3} \(G\) is proximinal for a semimetric space.
\end{enumerate}
\end{theorem}

\begin{proof}
\(\ref{t2.1:s1} \Rightarrow \ref{t2.1:s2}\). Let statement \ref{t2.1:s1} hold. Write \(X = A \cup B\). Our goal is to construct a metric \(d \colon X \times X \to [0, \infty)\) such that \(G = G_{X, d}(A, B)\).

Suppose first that \(G\) is empty. Then \(A\) and \(B\) are infinite by statement \ref{t2.1:s1}. 

Let us denote by \(\NN\) the set of all strictly positive integer numbers. Since \(A\) and \(B\) are infinite, there is a surjective mapping \(\Phi \colon X \to \NN\) such that
\begin{equation}\label{t2.1:e0}
\Phi(A) = 2\NN = \{2, 4, 6, \ldots\}, \quad \text{and} \quad \Phi(B) = 2\NN - 1 = \{1, 3, 5, \ldots\}.
\end{equation}

Let us define a function \(d_1 \colon X \times X \to [0, \infty)\) by the rule
\begin{equation}\label{t2.1:e1}
d_1(x, y) = \begin{cases}
0 & \text{if } x = y,\\
1 & \text{if \(x \neq y\) but \(\Phi(x) = \Phi(y)\)},\\
\max\left\{1 + \frac{1}{\Phi(x)}, 1 + \frac{1}{\Phi(y)}\right\} & \text{if } \Phi(x) \neq \Phi(y).
\end{cases}
\end{equation}
We claim that \(d_1\) is an ultrametric on \(X\) and the sets \(A\), \(B\) are proximinal sets in \((X, d_1)\). 

Indeed, it follows directly from~\eqref{t2.1:e1} that \(d_1\) is a symmetric mapping and \(d_1(x, y) = 0\) holds if and only if \(x = y\). Thus, \(d_1\) is an ultrametric on \(X\) if and only if we have the strong triangle inequality
\begin{equation}\label{t2.1:e5}
d_1(x, y) \leqslant \max \{d_1(x, z), d_1(z, y)\}
\end{equation}
for all \(x\), \(y\), \(z \in X\). 

Inequality~\eqref{t2.1:e5} evidently holds if \(x = y\). Let \(x\) and \(y\) be distinct points of \(X\). To prove~\eqref{t2.1:e5} suppose first that \(\Phi(x) = \Phi(y)\). Then \eqref{t2.1:e1} implies the equality \(d_1(x, y) = 1\) and the inequality
\[
\max \{d_1(x, z), d_1(z, y)\} \geqslant 1.
\]
(If the last inequality is false, then, by \eqref{t2.1:e1}, \(d_1(x, z) = d_1(z, y) = 0\) and, consequently, we have \(x = y\) that contradicts \(x \neq y\).)

For the case when \(\Phi(x) \neq \Phi(y)\) but \(\Phi(z) = \Phi(x)\) or \(\Phi(z) = \Phi(y)\), inequality~\eqref{t2.1:e5} can be written as
\[
\max\left\{1 + \frac{1}{\Phi(x)}, 1 + \frac{1}{\Phi(y)}\right\} \leqslant \max\left\{1 + \frac{1}{\Phi(x)}, 1 + \frac{1}{\Phi(y)}, 1\right\}
\]
which is obviously true.

To complete the proof of inequality~\eqref{t2.1:e5}, it suffices to note that this inequality is equivalent to the obvious inequality
\[
\max\left\{1 + \frac{1}{\Phi(x)}, 1 + \frac{1}{\Phi(y)}\right\} \leqslant \max\left\{1 + \frac{1}{\Phi(x)}, 1 + \frac{1}{\Phi(y)}, 1 + \frac{1}{\Phi(z)}\right\}
\]
whenever \(\Phi(x)\), \(\Phi(y)\) and \(\Phi(z)\) are pairwise distinct. Thus, \(d_1 \colon X \times X \to [0, \infty)\) is an ultrametric on \(X\).

Let us prove that \(A\) is proximinal. Let \(b_0 \in B\). By \eqref{t2.1:e1}, we obtain
\begin{equation}\label{t2.1:e6}
\dist(b_0, A) = \inf_{x \in A} \left\{\max\left\{1 + \frac{1}{\Phi(x)}, 1 + \frac{1}{\Phi(b_0)}\right\}\right\}.
\end{equation}
Now, using~\eqref{t2.1:e0}, we can find \(a_0 \in A\) such that \(\Phi(a_0) > \Phi(b_0)\). The last inequality and~\eqref{t2.1:e6} imply
\[
d_1(a_0, b_0) = \max\left\{1 + \frac{1}{\Phi(a_0)}, 1 + \frac{1}{\Phi(b_0)}\right\} = 1 + \frac{1}{\Phi(b_0)}
\]
and
\[
\dist(b_0, A) \geqslant \max\left\{1 + \frac{1}{\Phi(a_0)}, 1 + \frac{1}{\Phi(b_0)}\right\} = d_1(a_0, b_0) \geqslant \dist(b_0, A).
\]
Thus, \(A\) is a proximinal subset in \((X, d)\). Similarly, we can prove that \(B\) is also proximinal in \((X, d)\). Moreover, \eqref{t2.1:e1} implies that
\[
\dist(A, B) = 1
\]
and \(d(x, y) > 1\) whenever \(x \in A\), \(y \in B\) or \(y \in A\), \(x \in B\). Thus, \(G = G_{X, d_1}(A, B)\) for \((X, d_1) \in \mathbf{UM}\) by Definition~\ref{d1.5}. 

Let us consider now the case when \(G\) is nonempty and define \(d_2 \colon X \times X \to [0, \infty)\), \(X = A \cup B\), by
\begin{equation}\label{t2.1:e2}
d_2(x, y) = \begin{cases}
0 & \text{if } x = y,\\
1 & \text{if } \{x, y\} \in E(G),\\
2 & \text{otherwise}.
\end{cases}
\end{equation}
It is clear that \(d_2\) is a metric. Since the distance set \(D(X, d_2)\) is finite, for every \(x_0 \in X\) and every nonempty \(Z \subseteq X\), we can find \(z_0 \in Z\) such that
\[
d_2(x_0, z_0) = \inf \{d_2(x_0, z) \colon z \in Z\},
\]
i.e., every nonempty subset of \(X\) is proximinal in \((X, d_2)\). Thus, \((A, B)\) is a proximinal pair in \((X, d_2)\). Since \(E(G) \neq \varnothing\) holds, from~\eqref{t2.1:e2} follows that
\[
\dist(A, B) = 1.
\]
The last equality, \eqref{t2.1:e2}, Definition~\ref{d1.2} and Definition~\ref{d1.5} imply now that \(G = G_{X, d_2}(A, B)\) holds.

\(\ref{t2.1:s2} \Rightarrow \ref{t2.1:s3}\). This implication is valid because every metric is a semimetric.

\(\ref{t2.1:s3} \Rightarrow \ref{t2.1:s1}\). Let \(A\) and \(B\) be disjoint proximinal sets in a semimetric space \((X, d)\) and let \(G = G_{X, d}(A, B)\). We must show that \ref{t2.1:s1} is valid. 

Suppose, on the contrary, that \(G\) is empty, \(E(G) = \varnothing\), but at least one from the sets \(A\) and \(B\) is finite. For definiteness, we can assume \(|A|\leqslant |B|\) that implies the finiteness of \(A\). Since \(B\) is proximinal, for every \(a \in A\) there is \(b^* = b^*(a) \in B\) such that \(d(a, b^*) = \dist (a, B)\). Now using \eqref{d1.1:e1} and \eqref{e1.2} we have
\[
\dist (A, B) = \inf_{a \in A} \dist (a, B).
\]
Since \(A\) is finite and nonempty, we can find \(a_0 \in A\) such that
\[
\inf_{a \in A} \dist (a, B) = \dist (a_0, B).
\]
Consequently, \(\dist (A, B) = d(a_0, b^*)\) holds with \(b^* = b^*(a)\). By Definition~\ref{d1.5}, the last equality imply \(\{a_0, b^*\} \in E(G)\), contrary to \(E(G) = \varnothing\).
\end{proof}

Theorem~\ref{t2.1} gives us the following corollary.

\begin{corollary}\label{c2.2}
A bipartite graph \(G\) is not isomorphic to any proximinal graph if and only if \(G\) is finite and empty.
\end{corollary}

\begin{proof}
If \(G\) is finite and empty, then \(G\) is not isomorphic to any proximinal graph by Theorem~\ref{t2.1}. 

Conversely, suppose that there is a proximinal graph \(H\) with parts \(A\) and \(B\) such that \(G\) and \(H\) are isomorphic. Let \(\Phi \colon V(H) \to V(G)\) be an isomorphism of \(H\) and \(G\). Then the bipartite graph \(G\) with parts \(\Phi(A)\) and \(\Phi(B)\) is also proximinal by Theorem~\ref{t2.1}. Using this theorem again, we obtain that either \(E(G) \neq \varnothing\) or \(E(G) = \varnothing\) but the parts \(\Phi(A)\) and \(\Phi(B)\) of \(G\) are infinite.
\end{proof}

Analyzing the proof of Theorem~\ref{t2.1}, we also obtain the following.

\begin{corollary}\label{c2.3}
Let \((X, d_1)\) be an arbitrary semimetric space. Then, for every nonempty \(G_{X, d_1}(A, B)\), there is a metric \(d_2 \colon X \times X \to [0, \infty)\) such that \(G_{X, d_1}(A, B) = G_{X, d_2}(A, B)\) and the cardinality of the distance set \(D(X, d_2)\) does not exceed three,
\begin{equation}\label{c2.3:e1}
|D(X, d_2)| \leqslant 3.
\end{equation}
\end{corollary}

\begin{proof}
It suffices to define \(d_2(x, y)\) by formula~\eqref{t2.1:e2} for all \(x\), \(y \in X\).
\end{proof}

\begin{remark}\label{r2.4}
This is easy to prove that the constant \(3\) is the best possible in inequality~\eqref{c2.3:e1}.
\end{remark}

In the remainder of this section, we briefly discuss of some relationships between morphisms of semimetric spaces and morphisms of the proximinal graphs generated by these spaces. 

Let \((X, d)\) and \((Y, \rho)\) be isometric semimetric spaces. Recall that \((X, d)\) and \((Y, \rho)\) are \emph{isometric} iff there is a bijection \(F \colon X \to Y\), an \emph{isometry} of \((X, d)\) and \((Y, \rho)\), such that the equality
\[
d(x, y) = \rho(F(x), F(y))
\]
holds for all \(x\), \(y \in X\). It is easy to see that a set \(A \subseteq X\) is proximinal in \((X, d)\) iff \(F(A)\) is proximinal in \((Y, \rho)\). Moreover, for every \(x \in X\) and each \(B \subseteq X\), a point \(b_0 \in B\) is a best approximation to \(x\) in \(B\) iff \(F(b_0)\) is a best approximation to \(F(x)\) in \(F(B)\), and the equality
\[
\inf\bigl\{d(x, y) \colon x \in A \text{ and } y \in B\bigr\} = 
\inf\bigl\{\rho(u, v) \colon u \in F(A) \text{ and } v \in F(B)\bigr\} 
\]
holds for all \(A\), \(B \subseteq X\). Hence, for all disjoint proximinal sets \(A\), \(B \subseteq X\), the sets \(F(A)\) and \(F(B)\) are disjoint and proximinal in \((Y, \rho)\), the proximinal graphs \(G_{X, d}(A, B)\) and \(G_{Y, \rho}(F(A), F(B))\) are isomorphic, and the restriction \(F|_{A \cup B}\) is an isomorphism of these graphs. The following proposition is a partial reversal of the last statement.

\begin{proposition}\label{p2.3}
Let \(H\) be a graph and let \(G = G(A, B)\) be a bipartite graph with fixed parts \(A\) and \(B\). Suppose \(H\) and \(G\) are isomorphic and \(\Phi \colon V(G) \to V(H)\) is an isomorphism of these graphs. Then \(H\) is bipartite with parts \(\Phi(A)\) and \(\Phi(B)\), and, in addition, \(G(A, B)\) is proximinal iff \(H = H(\Phi(A), \Phi(B))\) is proximinal. Moreover, there are a metric \(d\) on the set \(X = A \cup B\) and a metric \(\rho\) on \(Y = \Phi(A) \cup \Phi(B)\) such that 
\[
G = G_{X, d}(A, B) \quad \text{and} \quad H = H_{Y, \rho}(\Phi(A), \Phi(B)),
\]
and \(\Phi \colon X \to Y\) is an isometry of \((X, d)\) and \((Y, \rho)\).
\end{proposition}

\begin{proof}
It follows directly from Definitions~\ref{d1.4} and~\ref{d1.6} that \(H\) is bipartite with parts \(\Phi(A)\) and \(\Phi(B)\).

Let \(G\) be proximinal. It was shown in the proof of Theorem~\ref{t2.1} that there is a metric \(d \colon X \times X \to [0, \infty)\) such that \(G = G_{X, d}(A, B)\) and \(X = A \cup B\). Let us denote by \(Y\) the vertex set of \(H\), and define \(\rho \colon Y \times Y \to [0, \infty)\) such that
\[
d(x, y) = \rho(\Phi(x), \Phi(y))
\]
for all \(x\), \(y \in X\). Then \(\rho\) is a metric on \(Y\) and the mapping \(\Phi \colon X \to Y\) is an isometry of the metric spaces \((X, d)\) and \((Y, \rho)\). Thus, we have the equality
\[
H = H_{Y, \rho}(\Phi(A), \Phi(B)).
\]

Analogously, if the graph \(H = H(\Phi(A), \Phi(B))\) is proximinal, then arguing as above and using the inverse isomorphism \(\Phi^{-1} \colon V(H) \to V(G)\) instead of \(\Phi \colon V(G) \to V(H)\), we can show that \(G = G(A, B)\) is also proximinal. 
\end{proof}

\begin{example}\label{ex2.3}
Let \((X, d)\) be a metric space, \(A\) be a non-closed subset of \(X\) and let \(b \notin A\) be a limit point of \(A\). Let us consider the bipartite graph \(G = G(A, B)\) with parts \(A\) and \(B = \{b\}\), and the edges set \(E(G)\) such that
\[
\bigl(\{a, b\} \in E(G)\bigr) \Leftrightarrow \bigl(d(a, b) = \dist(A, B)\bigr)
\]
whenever \(a \in A\) and \(b \in B\). Then \(G\) is an empty graph, \(A\) is infinite and \(B\) is finite. By Theorem~\ref{t2.1}, \(G\) is not a proximinal graph for any semimetric space, but nevertheless, using this theorem and Corollary~\ref{c2.2}, we can find a metric space \((Y, \rho)\) and a proximinal graph \(H = H_{Y, \rho}(A, B)\) such that \(H\) and \(G\) are isomorphic.
\end{example}

Let \((X, d)\) be a semimetric space and let \(Z \subseteq X\). A mapping \(\Phi \colon Z \to Z\) is said to be \emph{nonexpansive} if the inequality 
\[
d(\Phi(x), \Phi(y)) \leqslant d(x, y)
\]
holds for all \(x\), \(y \in Z\). A mapping \(F \colon A \cup B \to A \cup B\), defined on the union of two nonempty sets \(A\) and \(B\), is said to be \emph{cyclic} if we have
\[
F(a) \in B \quad \text{and} \quad F(b) \in A
\]
for all \(a \in A\) and \(b \in B\).

Let \(G\) and \(H\) be graphs. Following~\cite{HN2004} we say that a mapping \(\Psi \colon V(G) \to V(H)\) is a \emph{homomorphism} of \(G\) and \(H\) if \(\{\Psi(u), \Psi(v)\} \in E(H)\) whenever \(\{u, v\} \in E(G)\). If \(G = H\), then a homomorphism \(\Psi \colon V(G) \to V(H)\) is said to be self-homomorphism. The next proposition was motivated by paper~\cite{EKV2005SM}.

\begin{proposition}\label{p2.7}
Let \(A\) and \(B\) be disjoint proximinal sets in a semimetric space \((X, d)\) and let \(F \colon A \cup B \to A \cup B\) be cyclic and nonexpansive. Then \(F\) is a self homomorphism of the proximinal graph \(G = G_{X, d}(A, B)\).
\end{proposition}

\begin{proof}
Let \(\{a, b\}\) be an edge of \(G\). We must show that \(\{F(a), F(b)\} \in E(G)\). Suppose that 
\begin{equation}\label{p2.7:e1}
a \in A \quad \text{and} \quad b \in B
\end{equation}
(the case when \(a \in B\) and \(b \in A\) is similar). From~\eqref{p2.7:e1} it follows that
\begin{equation}\label{p2.7:e2}
F(a) \in B \quad \text{and} \quad F(b) \in A,
\end{equation}
because \(F\) is cyclic. In addition, we have
\begin{equation}\label{p2.7:e3}
d(a, b) \geqslant d(F(a), F(b)),
\end{equation}
because \(F\) is nonexpansive. Now \eqref{p2.7:e1}, \eqref{p2.7:e3} and Definition~\ref{d1.5} imply
\[
\dist(A, B) = d(a, b) \geqslant d(F(a), F(b)) \geqslant \dist(A, B).
\]
Hence, the equality
\[
\dist(A, B) = d(F(a), F(b))
\]
holds whenever \(\{a, b\} \in E(G)\). Using \eqref{p2.7:e3} and Definition~\ref{d1.5} again, we obtain the relationship
\[
\{F(a), F(b)\} \in E(G)
\]
for every \(\{a, b\} \in E(G)\). Thus, \(F\) is a self homomorphism of \(G\).
\end{proof}

\begin{corollary}\label{c2.7}
Let \(A\), \(B\) and \(F\) satisfy the conditions of Proposition~\ref{p2.7}. If \((a_0, b_0)\) is a best proximity pair for \(A\) and \(B\), \(\{a_0, b_0\} \in E(G_{X, d}(A, B))\), then we have the equality
\[
d(a_0, b_0) = d(F^n(a_0), F^n(b_0))
\]
for every \(n \in \NN\), where 
\[
F^1(a_0) = F(a_0), \quad F^1(b_0) = F(b_0) \quad \text{for} \quad n = 1
\]
and
\[
F^n(a_0) = F(F^{n-1}(a_0)), \quad F^n(b_0) = F(F^{n-1}(b_0)) \quad \text{for} \quad n \geqslant 2.
\]
\end{corollary}

\begin{proof}
It follows from Proposition~\ref{p2.7} and the definitions of proximinal graphs and graph homomorphisms by induction on \(n\).
\end{proof}

\section{Proximinal graphs for ultrametric spaces}

In the present section we investigate the structure of the proximinal graphs for ultrametric spaces. The next result is a part of Theorem~2.6 from~\cite{CDL2021pNUAA}.

\begin{theorem}\label{t2.3}
Let \((A, B)\) be a proximinal pair in \((X, d) \in \mathbf{UM}\). Then the following statements are equivalent:
\begin{enumerate}
\item\label{t2.3:s1} The inequality \(\diam (B)\leq \dist(A, B)\) holds.
\item\label{t2.3:s2} The sets \(A_0 \subseteq A\) and \(B_0 \subseteq B\) are proximinal subsets of \(X\), and the equality \(B_0 = B\) holds, and every \((a, b) \in A_0 \times B_0\) is a best proximity pair for the sets \(A\) and \(B\).     
\end{enumerate}
\end{theorem}

For the case \(A \cap B = \varnothing\), Theorem~\ref{t2.3} implies the following.

\begin{proposition}\label{p2.5}
Let \(G = G_{X, d}(A, B)\) be a proximinal for \((X, d) \in \mathbf{UM}\). Then the following statements are equivalent:
\begin{enumerate}
\item\label{p2.5:s1} The inequality \(\diam (B) \leqslant \dist(A, B)\) holds.
\item\label{p2.5:s2} \(G\) is nonempty and \(G'\) is a complete bipartite graph such that \(B \subseteq V(G')\).
\end{enumerate}
\end{proposition}

\begin{proof}
\(\ref{p2.5:s1} \Rightarrow \ref{p2.5:s2}\). Let \ref{p2.5:s1} hold. Then, by statement~\ref{t2.3:s2} of Theorem~\ref{t2.3}, the sets \(A_0\) and \(B_0\) are proximinal and, consequently, nonempty (see Remark~\ref{r1.7}). Using Remark~\ref{r1.4} and statement~\ref{t2.3:s2} of Theorem~\ref{t2.3}, we see that \(G\) is a nonempty graph. Hence, \(G'\) is correctly defined. We must show that \(G'\) is a complete bipartite graph and \(B \subseteq V(G')\) holds. 

Since \(G\) is a bipartite graph with parts \(A\) and \(B\), the inclusions \(A_0 \subseteq A\), \(B_0 \subseteq B\), and the equality
\begin{equation}\label{p2.5:e1}
A_0 \cup B_0 = V(G')
\end{equation}
(see Remark~\ref{r1.7}) imply that \(G'\) is a bipartite graph with parts \(A_0\) and \(B_0\). 
From statement~\ref{p2.5:s1} of the present proposition it follows statement~\ref{t2.3:s1} of Theorem~\ref{t2.3}. Consequently, we have \(\{a, b\} \in V(G')\) for all \(a \in A_0\) and \(b \in B_0\). Thus, \(G'\) is a complete bipartite graph. The inclusion \(B \subseteq V(G')\) holds because \(B_0 \subseteq V(G')\) by~\eqref{p2.5:e1} and we have \(B_0 = B\) by statement~\ref{t2.3:s2} of Theorem~\ref{t2.3}.

\(\ref{p2.5:s2} \Rightarrow \ref{p2.5:s1}\). Let \ref{p2.5:s2} hold. Let us consider an arbitrary \(z \in A_0\). As was noted above, \(G'\) is bipartite with parts \(A_0\) and \(B_0\). Now, using the condition \((X, d) \in \mathbf{UM}\), and statement~\ref{p2.5:s2}, and Definition~\ref{d1.5}, we can prove that
\begin{align*}
\diam(B) &= \sup_{x, y \in B} d(x, y) \leqslant \sup_{x, y \in B} \max \{d(x, z), d(z, y)\} \leqslant \dist(A, B). \qedhere
\end{align*}
\end{proof}

\begin{corollary}\label{c2.10}
Let \(G = G_{Y, \rho} (C, D)\) be a proximinal graph for \((Y, \rho) \in \mathbf{UM}\). Then the following statements are equivalent:
\begin{enumerate}
\item \label{c2.10:s1} \(G\) is connected.
\item \label{c2.10:s2} The inequality 
\begin{equation}\label{c2.10:e1}
\diam (C \cup D) \leqslant \dist (C, D)
\end{equation}
holds.
\item \label{c2.10:s3} \(G\) is a complete bipartite graph.
\end{enumerate}
\end{corollary}

\begin{proof}
\(\ref{c2.10:s1} \Rightarrow \ref{c2.10:s2}\). Let \(G\) be connected. Then, for any two distinct \(x\), \(y \in V(G)\), there is a path \(P \subseteq G\) such that 
\begin{equation*}
V(P) = \{x_0, x_1, \ldots, x_k\}, \quad E(P) = \bigl\{\{x_0, x_1\}, \ldots, \{x_{k-1}, x_k\}\bigr\},
\end{equation*}
where \(k \geqslant 1\) and \(x_0 = x\), \(x_k = y\). Since \(E(P) \subseteq E(G)\) holds and \(G\) is bipartite with parts \(C\) and \(D\), Definition~\ref{d1.5} implies the equality
\begin{equation}\label{c2.10:e2}
\rho(x_i, x_{i+1}) = \dist(C, D)
\end{equation}
for every \(i \in \{0, \ldots, k-1\}\). Now using~\eqref{c2.10:e2} and the strong triangle inequality, we obtain
\begin{equation}\label{c2.10:e3}
\rho(x, y) = \rho(x_0, x_k) \leqslant \sup_{0 \leqslant i \leqslant k-1} \rho(x_i, x_{i+1}) = \dist(C, D)
\end{equation}
by induction on \(k\). The equality 
\[
\diam (C \cup D) = \sup_{x, y \in C \cup D} \rho(x, y)
\]
and \eqref{c2.10:e3} imply \eqref{c2.10:e1}.

\(\ref{c2.10:s2} \Rightarrow \ref{c2.10:s3}\). Let \ref{c2.10:s2} hold. Then we have
\[
\diam (C) \leqslant \dist(C, D) \quad \text{and} \quad \diam (D) \leqslant \dist(C, D).
\]
These inequalities and Theorem~\ref{t2.3} imply that \(C_0 = C\) and \(D_0 = D\), and that every \((x, y) \in C \times D\) is a best proximity pair for \((C, D)\). Consequently, \(G\) is a complete bipartite graph by Definition~\ref{d1.4}.

\(\ref{c2.10:s3} \Rightarrow \ref{c2.10:s1}\). This is trivially valid.
\end{proof}

\begin{figure}[ht]
\begin{tikzpicture}[scale=1,
arrow/.style = {-{Stealth[length=5pt]}, shorten >=2pt}]
\def\xx{1.8cm}
\def\yy{1cm}
\coordinate (A1) at (0*\xx, \yy);
\coordinate (A2) at (1*\xx, \yy);
\coordinate (A3) at (2*\xx, \yy);

\coordinate (B1) at (0*\xx, -\yy);
\coordinate (B2) at (1*\xx, -\yy);
\coordinate (B3) at (2*\xx, -\yy);

\draw [fill, black] (A1) circle (2pt);
\draw [fill, black] (A2) circle (2pt);
\draw [fill, black] (A3) circle (2pt);
\draw [fill, black] (B1) circle (2pt);
\draw [fill, black] (B2) circle (2pt);
\draw [fill, black] (B3) circle (2pt);

\draw (A1) -- (B3) -- (A2) -- (B1) -- (A1) -- (B2) -- (A3) -- (B3);
\draw (A2) -- (B2);
\draw (A3) -- (B1);
\draw (2*\xx+1cm, \yy) node {\(K_{3,3}\)};
\end{tikzpicture}
\caption{}\label{fig2}
\end{figure}

\begin{example}\label{ex3.4}
The proximinal graph \(G_{X, d}(A, B)\), which was constructed in Example~\ref{ex1.10}, is a regular bipartite graph of degree \(3\) and has \(8\) vertices. The graph \(K_{3,3}\) (see Figure~\ref{fig2}) is the unique, up to isomorphism, regular complete bipartite graph of degree \(3\). Since \(|V(K_{3,3})| = 6\), the graphs \(G_{X, d}(A, B)\) and \(K_{3,3}\) are not isomorphic. Hence, by Corollary~\ref{c2.10}, \(G_{X, d}(A, B)\) is not isomorphic to any proximinal graph \(G_{Y, \rho}(A, B)\) for \((Y, \rho) \in \mathbf{UM}\).
\end{example}

\begin{example}\label{ex3.6}
Let \((Y, \rho)\) be the ultrametric space endowed with the trivial metric 
\[
\rho(x, y) = \begin{cases}
0 & \text{if } x = y,\\
1 & \text{otherwise},
\end{cases}
\]
and let \(K_{n, m}\) be a complete bipartite graph. If \(C\) and \(D\) are disjoint subsets of \(Y\) such that \(|C| = n\) and \(|D| = m\), then \(C\) and \(D\) are proximinal in \((Y, \rho)\), and \(K_{n, m}\) is isomorphic to the proximinal graph \(G_{Y, \rho}(C, D)\).
\end{example}

To characterize the proximinal graphs for general ultrametric spaces, we will use some properties of closed ultrametric balls.

\begin{lemma}\label{l4.3}
Let \((X, d)\) be an ultrametric space. Then, for every \(\ol{B}_r(c) \in \bB_{X}\) and every \(a \in \ol{B}_r(c)\), we have \(\ol{B}_r(c) = \ol{B}_r(a)\).
\end{lemma}

For the proof of this lemma see, for example, Proposition~18.4 \cite{Sch1985}.

The next corollary and Lemma~\ref{l4.6} are modifications of Corollary~4.5 and, respectively, of Lemma~4.6 from~\cite{DS2022TA} to the case of closed balls.

\begin{corollary}\label{c4.4}
Let \((X, d)\) be an ultrametric space. Then, for all \(\ol{B}_{r_1}(c_1)\), \(\ol{B}_{r_2}(c_2) \in \bB_{X}\), we have
\begin{equation}\label{c4.4:e1}
\ol{B}_{r_1}(c_1) \subseteq \ol{B}_{r_2}(c_2)
\end{equation}
whenever \(\ol{B}_{r_1}(c_1) \cap \ol{B}_{r_2}(c_2) \neq \varnothing\) and \(0 \leqslant r_1 \leqslant r_2 < \infty\). In particular, for every \(r \geqslant 0\) and all \(\ol{B}_{r}(c_1)\), \(\ol{B}_{r}(c_2) \in \bB_{X}\),  we have
\begin{equation}\label{c4.4:e2}
\ol{B}_{r}(c_1) = \ol{B}_{r}(c_2)
\end{equation}
whenever \(\ol{B}_{r}(c_1) \cap \ol{B}_{r}(c_2) \neq \varnothing\).
\end{corollary}

\begin{proof}
Let \(\ol{B}_r(c_1) \cap \ol{B}_r(c_2) \neq \varnothing\) and let \(r \geqslant 0\) be fixed. Then, by Lemma~\ref{l4.3}, we obtain the equalities 
\[
\ol{B}_r(c_1) = \ol{B}_r(a)\quad \text{and} \quad \ol{B}_r(c_2) = \ol{B}_r(a)
\]
for every \(a \in \ol{B}_r(c_1) \cap \ol{B}_r(c_2)\). Equality~\eqref{c4.4:e2} follows.

If we have \(0 \leqslant r_1 \leqslant r_2 < \infty\) and \(\ol{B}_{r_1}(c_1) \cap \ol{B}_{r_2}(c_2) \neq \varnothing\), then \eqref{c4.4:e2} implies the equality \(\ol{B}_{r_2}(c_1) = \ol{B}_{r_2}(c_2)\). Now \eqref{c4.4:e1} follows from the last equality and the inclusion \(\ol{B}_{r_1}(c_1) \subseteq \ol{B}_{r_2}(c_1)\).
\end{proof}

\begin{lemma}\label{l4.6}
Let \((X, d) \in \mathbf{UM}\) and let \(r \geqslant 0\) be fixed. Then there is a set \(C \subseteq X\) such that
\begin{equation}\label{l4.6:e1}
X = \bigcup_{c \in C} \ol{B}_r(c)
\end{equation}
and \(\ol{B}_r(c_1) \cap \ol{B}_r(c_2) = \varnothing\) whenever \(c_1\) and \(c_2\) are distinct points of \(C\).
\end{lemma}

\begin{proof}
Let us define a binary relation \(\myeq\) on the set \(X\) as
\[
(x_1 \myeq x_2) \Leftrightarrow (\ol{B}_r(x_1) \cap \ol{B}_r(x_2) \neq \varnothing).
\]
It is clear that the relation \(\myeq\) is reflexive and symmetric,
\[
x \myeq x, \quad \text{and} \quad (y \myeq z) \Leftrightarrow (z \myeq y)
\]
for all \(x\), \(y\), \(z \in X\). Moreover, Corollary~\ref{c4.4} implies the the transitivity of \(\myeq\). Hence, \(\myeq\) is an equivalence relation on \(X\).

Using the well-known one-to-one correspondence between the equivalence relations and partitions (see, for example, \cite[Chapter~II, \S~5]{KurMost}), we can find a partition \(\widetilde{X} = \{X_i \colon i \in I\}\) of the set \(X\) generated by relation \(\myeq\). It means that, for all \(a\), \(b \in X\), we have \(a \myeq b\) iff there is \(i \in I\) such that \(a\), \(b \in X_i\) and
\begin{equation*}
\bigcup_{i \in I} X_i = X, \quad X_i \cap X_j = \varnothing \quad \text{and} \quad X_k \neq \varnothing
\end{equation*}
whenever \(i\), \(j\), \(k \in I\) and \(i \neq j\). Now if we define \(C \subseteq X\) as a system of distinct representatives for \(\widetilde{X}\),
\[
C = \{c_i \colon i \in I \text{ and } c_i \in X_i\},
\]
then, using Corollary~\ref{c4.4}, we obtain the equality
\begin{equation}\label{l4.6:e3}
\ol{B}_r(c_i) = \bigcup_{c \in X_i} \ol{B}_r(c)
\end{equation}
for every \(i \in I\). In addition, from the definition of \(\myeq\) it follows that \(\ol{B}_r(c_{i_1})\) and \(\ol{B}_r(c_{i_2})\) are disjoint for all distinct \(i_1\), \(i_2 \in I\). To complete the proof it suffices to note that \eqref{l4.6:e1} follows from the definition of \(\widetilde{X}\) and \eqref{l4.6:e3}.
\end{proof}

\begin{theorem}\label{t2.9}
Let \(G\) be a bipartite graph with fixed parts \(A\) and \(B\). Then the following statements are equivalent:
\begin{enumerate}
\item \label{t2.9:s1} Either \(G\) is nonempty and \(G'\) is the disjoint union of complete bipartite graphs, or \(G\) is empty, but the sets \(A\) and \(B\) are infinite.
\item \label{t2.9:s2} \(G\) is proximinal for an ultrametric space \((X, d)\) with \(X = A \cup B\).
\item \label{t2.9:s3} \(G\) is proximinal for an ultrametric space.
\end{enumerate}
\end{theorem}

\begin{proof}
\(\ref{t2.9:s1} \Rightarrow \ref{t2.9:s2}\). Let \ref{t2.9:s1} hold. If \(G\) is empty, but \(A\) and \(B\) are infinite, then it was shown in the proof of Theorem~\ref{t2.1} that \(G = G_{X, d_1}(A, B)\) when \(X = A\cup B\) and \(d_1\) is the ultrametric defined by formula~\eqref{t2.1:e1}.

Let us consider the case of nonempty \(G\). Then there is a family \(\mathcal{F} = \{G_i \colon i \in I\}\) (where \(I\) is a set of indexes) such that every \(G_i\) is a complete bipartite graph with parts \(A_i = A \cap V(G_i)\) and \(B_i = B \cap V(G_i)\), and
\begin{gather}
\label{t2.9:e1}
V(G') = \bigcup_{i \in I} V(G_i),\\
\label{t2.9:e2}
E(G') = \bigcup_{i \in I} E(G_i),\\
\label{t2.9:e3}
V(G_{i_1}) \cap V(G_{i_2}) = \varnothing
\end{gather}
for all different \(i_1\), \(i_2 \in I\). Write \(X := A \cup B\) and define \(d_2 \colon X \times X \to [0, \infty)\) as
\[
d_2(x, y) = \begin{cases}
0 & \text{if } x = y,\\
1 & \text{if \(x \neq y\) and there is \(i \in I\) such that \(x\), \(y \in V(G_i)\)} ,\\
2 & \text{otherwise}.
\end{cases}
\]
Reasoning similar (but much simpler) to those used in the proof of Theorem~\ref{t2.1} show that \((X, d_2) \in \mathbf{UM}\) and \(G = G_{X, d_2}(A, B)\).

\(\ref{t2.9:s2} \Rightarrow \ref{t2.9:s3}\). This implication is evidently valid.

\(\ref{t2.9:s3} \Rightarrow \ref{t2.9:s1}\). Let \(G = G_{X, d}(A, B)\) hold for \((X, d) \in \mathbf{UM}\). If \(G\) is empty, then the sets \(A\) and \(B\) are infinite by Theorem~\ref{t2.1}. Let us consider now the case when \(E(G) \neq \varnothing\).

It follows directly from Definitions~\ref{d1.1}, \ref{d1.2} and \ref{d1.5} that, for \(Z \subseteq X\), the graph \(G\) is proximinal for the ultrametric space \((Z, d|_{Z \times Z})\) whenever \(A \cup B \subseteq Z\), where \(d|_{Z \times Z}\) is the restriction of \(d \colon X \times X \to [0, \infty)\) on \(Z \times Z\). Thus, without loss of generality, we assume \(X = A \cup B\). Write \(r := \dist (A, B)\). Using Lemma~\ref{l4.6}, we can find \(C \subseteq X\) such that
\begin{equation}\label{t2.9:e4}
X = \bigcup_{c \in C} \ol{B}_r(c)
\end{equation}
and 
\begin{equation}\label{t2.9:e5}
\ol{B}_r(c_1) \cap \ol{B}_r(c_2) = \varnothing
\end{equation}
whenever \(c_1\), \(c_2\) are different points of \(C\). 

Let us define \(C_1 \subseteq C\) by the rule:
\begin{itemize}
\item a point \(c \in C\) belongs to \(C_1\) iff both sets \(\ol{B}_r(c) \cap A\) and \(\ol{B}_r(c) \cap B\) are nonempty.
\end{itemize}
We claim that the equality
\begin{equation}\label{t2.9:e6}
V(G') = \bigcup_{c \in C_1} \ol{B}_r(c)
\end{equation}
holds and, moreover, the membership relation \(\{a, b\} \in E(G')\) is valid iff there is \(c \in C_1\) such that \(a\), \(b \in \ol{B}_r(c)\) and 
\begin{equation}\label{t2.9:e6.1}
\{a, b\} \cap A \neq \varnothing \neq \{a, b\} \cap B. 
\end{equation}

Let us prove equality~\eqref{t2.9:e6}. If \(a \in V(G')\), then there is \(b \in V(G')\) such that \(\{a, b\} \in E(G)\) (see Remark~\ref{r1.5}). Consequently, 
\begin{equation}\label{t2.9:e7}
d(a, b) = \dist (A, B) = r
\end{equation}
holds. Using \eqref{t2.9:e4} and \eqref{t2.9:e5}, we can find a unique \(c \in C\) such that \(a \in \ol{B}_r(c)\). Corollary~\ref{c4.4} implies the equality
\begin{equation}\label{t2.9:e8}
\ol{B}_r(c) = \ol{B}_r(a).
\end{equation}
The definition of closed balls, \eqref{t2.9:e7} and \eqref{t2.9:e8} imply that \(\{a, b\} \subseteq \ol{B}_r(a)\). Now from \(\{a, b\} \in E(G)\) and \eqref{t2.9:e8} follow \(c \in C_1\). Hence, we have the inclusion
\begin{equation}\label{t2.9:e9}
V(G') \subseteq \bigcup_{c \in C_1} \ol{B}_r(c).
\end{equation}
To prove the converse inclusion,
\begin{equation}\label{t2.9:e10}
V(G') \supseteq \bigcup_{c \in C_1} \ol{B}_r(c),
\end{equation}
we consider arbitrary points \(c^* \in C_1\) and \(x^* \in \ol{B}_r(c^*)\). By definition of \(C_1\), we can find \(a^* \in A\) and \(b^* \in B\) such that \(a^*\) and \(b^*\) belong to \(\ol{B}_r(c^*)\). We may assume first that \(x^* \in A\). (Recall that \(X\) is the union of disjoint sets \(A\) and \(B\).) Then we obtain
\begin{equation}\label{t2.9:e11}
\dist (A, B) \leqslant d(x^*, b^*)
\end{equation}
and, by strong triangle inequality,
\begin{equation}\label{t2.9:e12}
d(x^*, b^*) \leqslant \max \bigl\{d(c^*, x^*), d(c^*, b^*)\bigr\} \leqslant r = \dist (A, B).
\end{equation}
Hence, the equality \(d(x^*, b^*) = \dist(A, B)\) holds, i.e., \(x^*\) belongs to \(V(G')\). The case when \(x^* \in B\) can be considered similarly. Thus, for every \(c^* \in C_1\) and every \(x^* \in \ol{B}_r(c^*)\), we have \(x^* \in V(G')\), that implies~\eqref{t2.9:e10}. Now~\eqref{t2.9:e6} follows from \eqref{t2.9:e9} and \eqref{t2.9:e10}.

Let us prove the validity of the equivalence 
\begin{equation}\label{t2.9:e13}
\bigl(\{a, b\} \in E(G')\bigr) \Leftrightarrow \bigl(\exists c \in C_1 \text{ such that } a, b \in \ol{B}_r(c) \text{ and (\ref{t2.9:e6.1}) holds}\bigr).
\end{equation}
If \(\{a, b\} \in E(G')\), then we evidently have \(b \in \ol{B}_r(a)\). By Lemma~\ref{l4.6}, the point \(a\) belong to \(\ol{B}_r(c)\) for some \(c \in C_1\). That implies \(\ol{B}_r(c) = \ol{B}_r(a)\) by Corollary~\ref{c4.4}.

Conversely, if (\ref{t2.9:e6.1}) holds and there is \(c \in C_1\) such that \(a\), \(b \in \ol{B}_r(c)\), then
\[
\dist(A, B) \leqslant d(a, b) \leqslant \max \bigl\{d(c, a), d(c, b)\bigr\} \leqslant r = \dist(A, B),
\]
that implies the equality
\[
d(a, b) = \dist (A, B).
\]
The membership \(\{a, b\} \in E(G')\) is valid.

To complete the proof, it suffices to show that \(G'\) is the disjoint union of complete bipartite graphs. Let us consider an arbitrary point \(c \in C_1\) and define an ultrametric space \((Y_c, \rho_c)\) as
\[
Y_c := \ol{B}_r(c) \quad \text{and} \quad \rho_c := d|_{\ol{B}_r(c) \times \ol{B}_r(c)},
\]
where \(d|_{\ol{B}_r(c) \times \ol{B}_r(c)}\) is the restriction of the ultrametric \(d\) on the closed ball \(\ol{B}_r(c)\), and write
\[
K_c = A \cap \ol{B}_r(c), \quad D_c = B \cap \ol{B}_r(c).
\]
Then \(K_c\) and \(D_c\) are disjoint proximinal subsets in \((Y_c, \rho_c)\). Let \(G_c\) be a graph such that \(V(G_c) = Y_c\) and 
\[
\bigl(\{a, b\} \in E(G_c)\bigr) \Leftrightarrow \bigl(\rho_c(a, b) = \dist(K_c, D_c)\bigr)
\]
is valid for all \(a\), \(b \in V(G_c)\). The graph \(G_c\) is proximinal for \((Y_c, \rho_c) \in \mathbf{UM}\) and has the parts \(K_c\) and \(D_c\). By Corollary~\ref{c2.10}, the graph \(G_c\) is a complete bipartite graph. Now from~\eqref{t2.9:e6} and \eqref{t2.9:e13} it follows that \(G'\) is the disjoint union of the graphs \(G_c\), \(c \in C_1\).
\end{proof}

\begin{corollary}\label{c2.13}
The following statements are equivalent for every nonempty graph \(G\):
\begin{enumerate}
\item\label{c2.13:s1} \(G'\) is the disjoint union of complete bipartite graphs.
\item\label{c2.13:s2} There is \((X, d) \in \mathbf{UM}\) such that \(G\) is isomorphic to a proximinal graph for \((X, d)\).
\end{enumerate}
\end{corollary}

\begin{example}\label{ex2.14}
Let \(f \colon [0, \infty) \to [0, \infty)\) and \(g \colon [0, \infty) \to [0, \infty)\) satisfy \(f(0) = g(0) = 0\) and \(f(x) > 0\), \(g(x) > 0\) for every \(x > 0\). Let us consider a semimetric space \((X, d)\) such that \(X = \CC\), where \(\CC\) is the set of all complex numbers \(z = x + iy\), and, for all \(z_1 = x_1 + iy_1\), \(z_2 = x_2 + iy_2\), 
\[
d(z_1, z_2) = f(|x_1 - x_2|) + g(|y_1 - y_2|).
\]
Write
\[
A = \{x + iy \in \CC \colon y = 0\} \quad \text{and} \quad B = \{x + iy \in \CC \colon y = 1\}.
\]
Then \((A, B)\) is a disjoint proximinal pair for \((X, d)\) and the equalities \(\dist(A, B) = g(1)\) and \(A_0 = A\), \(B_0 = B\) hold. Moreover, for every \(z_1 = x_1 + iy_1 \in \CC\), the point \(x_1\) is the unique best approximation to \(z_1\) in \(A\) and \(x_1 + i\) is the unique best approximation to \(z_1\) in \(B\).

Let us define a graph \(G\) as
\[
V(G) := A \cup B \quad \text{and} \quad E(G) := \bigl\{\{x, x+i\} \colon x \in \RR\bigr\},
\]
where \(\RR\) is the set of all real numbers. Then \(G\) is a proximinal graph for the semimetric space \((X, d)\). Since \(G\) is a disjoint union of complete two-point graphs \(K_2\), by Theorem~\ref{t2.9}, there is an ultrametric space \((Y, \rho)\) such that \(G\) is proximinal for \((Y, \rho)\) and has the parts \(A\) and~\(B\).
\end{example}

\begin{remark}\label{r2.5}
Ultrametric \(d_1\), that we use in the proof of Theorem~\ref{t2.1}, is obtained by ``blow up'' of the ultrametric from Example~2.2 of paper \cite{CDL2021pNUAA}. Ultrametrics of this type were first constructed by Delhomm\'{e}, Laflamme, Pouzet and Sauer \cite[Proposition 2]{ref201}. Similar constructions are often useful in the study of various topological and geometrical properties of ultrametric spaces \cite{BDS2021pNUAA, DDP2011pNUAA, DS2022TA, ref205} and have a natural generalization to the Priess---Crampe and Ribenboim Ultrametrics with totally ordered range sets (see \cite[Proposition 4.10]{ref206}).
\end{remark}

\section{From proximinal to farthest and back}

In the second section of the paper, we considered the bipartite graphs \(G(A, B)\), whose parts \(A\) and \(B\) are disjoint proximinal subsets of a semimetric space \((X, d)\), and vertices \(a \in A\) and \(b \in B\) are adjacent if and only if 
\begin{equation*}
d(a, b) = \inf_{\substack{x \in A \\ y \in B}} d(x, y).
\end{equation*}
Below we discuss the bipartite graphs \(G(A, B)\), whose parts \(A\) and \(B\) are arbitrary disjoint nonempty subsets of a semimetric space \((X, d)\) such that \(a \in A\) and \(b \in B\) are adjacent iff
\begin{equation}\label{e3.2}
d(a, b) = \sup_{\substack{x \in A \\ y \in B}} d(x, y)
\end{equation}
and show that these graphs are isomorphic to proximinal graphs. Let us start with a formal definition.

\begin{definition}\label{d3.1}
A graph \(G\) is \emph{farthest} if \(G\) is a bipartite graph with some fixed parts \(A\) and \(B\), and there is \((X, d) \in \mathbf{SM}\) such that \(A\) and \(B\) are disjoint subsets of \(X\), and vertices \(a \in A\) and \(b \in B\) are adjacent iff \eqref{e3.2} holds. In this case we say that \(G\) is farthest for \((X, d)\).
\end{definition}

The following theorem is an analog of Theorem~\ref{t2.1}.

\begin{theorem}\label{t3.2}
Let \(G\) be a bipartite graph with fixed parts \(A\) and \(B\). Then the following statements are equivalent:
\begin{enumerate}
\item \label{t3.2:s1} Either \(G\) is nonempty or \(G\) is empty but at least one from the parts \(A\), \(B\) is infinite.
\item \label{t3.2:s2} \(G\) is farthest for a metric space.
\item \label{t3.2:s3} \(G\) is farthest for a semimetric space.
\end{enumerate}
\end{theorem}

\begin{proof}
Let us suppose first that \(G\) is a nonempty graph.

\(\ref{t3.2:s1} \Rightarrow \ref{t3.2:s2}\). If \ref{t3.2:s1} holds, then we write
\begin{equation}\label{t3.2:e1}
X = A \cup B
\end{equation}
and define \(d \colon X \times X \to [0, \infty)\) by 
\begin{equation}\label{t3.2:e2}
d(x, y) = \begin{cases}
0 & \text{if } x = y,\\
2 & \text{if } \{x, y\} \in E(G),\\
1 & \text{otherwise}.
\end{cases}
\end{equation}
It follows directly from \eqref{t3.2:e2} that \(d\) is a metric on \(X\) and
\[
\sup_{\substack{x \in A \\ y \in B}} d(x, y) = 2.
\]
The last equality, \eqref{t3.2:e2} and Definition~\ref{d3.1} imply that \(G\) is farthest for the metric space \((X, d)\) and has the parts \(A\) and \(B\).

\(\ref{t3.2:s2} \Rightarrow \ref{t3.2:s3}\). This is valid because every metric space is semimetric.

\(\ref{t3.2:s3} \Rightarrow \ref{t3.2:s1}\). By our supposition, \(G\) is a nonempty bipartite graph, that implies \ref{t3.2:s1}.

Let us consider the case when \(G\) is empty, \(E(G) = \varnothing\). 

\(\ref{t3.2:s1} \Rightarrow \ref{t3.2:s2}\). Let \ref{t3.2:s1} hold. Then, without loss of generality, we may assume that \(B\) is an infinite set. Let us consider a partition \(\{B_i \colon i \in \NN\}\) of the set \(B\) on the disjoint nonempty sets \(B_i\), where \(\NN\) is the set of all positive integers. Let us define the set \(X\) by \eqref{t3.2:e1} and let \(d \colon X \times X \to [0, \infty)\) be a symmetric mapping which satisfies
\begin{equation}\label{t3.2:e3}
d(x, y) = \begin{cases}
0 & \text{if } x = y,\\
1 + \dfrac{i}{i+1} & \text{if } x \in A \text{ and } y \in B_i \text{ for some } i \in \NN,\\
1 & \text{otherwise}.
\end{cases}
\end{equation}
Then \((X, d)\) is a metric space such that
\[
\sup_{\substack{x \in A \\ y \in B}} d(x, y) = \sup_{i \in \NN} \left(1 + \dfrac{i}{i+1}\right) = 2.
\]
In addition, \eqref{t3.2:e3} implies the inequality \(d(x, y) < 2\) for all \(x\), \(y \in X\). Consequently, \(G\) is farthest for \((X, d) \in \mathbf{M}\) and has the parts \(A\) and \(B\).

\(\ref{t3.2:s2} \Rightarrow \ref{t3.2:s3}\). This implication is obviously true.

\(\ref{t3.2:s3} \Rightarrow \ref{t3.2:s1}\). Let \(G\) be a farthest graph for \((X, d) \in \mathbf{SM}\) and let \(A\), \(B \subseteq X\) be the parts of \(G\). We must show that at least one from the sets \(A\) and \(B\) is infinite. Suppose contrary that both sets \(A\) and \(B\) are finite. Then we can find \(a_0 \in A\) and \(b_0 \in B\) such that
\[
d(a_0, b_0) = \sup_{\substack{x \in A \\ y \in B}} d(x, y).
\]
By Definition~\ref{d3.1}, \(\{a_0, b_0\}\) is an edge of \(G\), contrary to \(E(G) = \varnothing\).
\end{proof}

Theorem~\ref{t3.2} gives us the following corollary.

\begin{corollary}\label{c3.3}
A bipartite graph is not isomorphic to any farthest graph if and only if \(G\) is finite and empty.
\end{corollary}

This corollary can be proved similarly to Corollary~\ref{c2.2} and we omit it here.

Using Theorems~\ref{t2.1} and \ref{t3.2} we also obtain the following.

\begin{corollary}\label{c3.4}
If a graph \(G\) is proximinal, then \(G\) is farthest but not conversely, in general.
\end{corollary}

Nevertheless, the proximinal graphs and the farthest ones coincide up to isomorphism.

\begin{proposition}\label{p3.4}
Let \(G\) be a graph. Then \(G\) is isomorphic to a farthest graph iff \(G\) is isomorphic to a proximinal graph.
\end{proposition}

\begin{proof}
It follows from Corollaries~\ref{c2.2} and \ref{c3.3}.
\end{proof}

\begin{example}\label{ex3.5}
Let \(G_1\) be a proximinal graph for \((X_1, d_1) \in \mathbf{SM}\) and let \(G_1\) have parts \(A_1\) and \(B_1\). Let us consider \(\rho_1 \colon X_1 \times X_1 \to [0, \infty)\) defined by
\[
\rho_1(x, y) = \begin{cases}
0 & \text{if } x = y,\\
\dfrac{1}{d_1(x, y)} & \text{if } x \neq y.
\end{cases}
\]
Then \(\rho_1\) is a semimetric on \(X_1\), and \(G_1\) is a farthest graph for \((X_1, \rho_1)\), and \(A_1\), \(B_1\) are the parts of \(G_1\). Moreover, for every \(x \in X_1 \setminus A_1\) and every \(y \in X_1 \setminus B_1\), there are \(a_x^1 \in A_1\) and \(b_y^1 \in B_1\) such that 
\[
\sup_{a \in A_1} \rho_1(x, a) = \rho_1(x, a_x^1) \quad \text{and} \quad \sup_{b \in B_1} \rho_1(y, b) = \rho(y, b_y^1).
\]

Conversely, let \(G_2\) be a farthest graph for \((X_2, \rho_2) \in \mathbf{SM}\), and let \(G_2\) have parts \(A_2\) and \(B_2\). Let us define a semimetric \(d_2\) on \(X_2\) as
\[
d_2(x, y) = \begin{cases}
0 & \text{if } x = y,\\
\dfrac{1}{\rho_2(x, y)} & \text{if } x \neq y.
\end{cases}
\]
Then \(G_2\) is proximinal for \((X_2, d_2)\) with parts \(A_2\) and \(B_2\) if and only if, for every \(x \in X_2 \setminus A_2\) and every \(y \in X_2 \setminus B_2\), there are \(a_x^2 \in A_2\) and \(b_y^2 \in B_2\) such that
\[
\sup_{a \in A_2} \rho_2(x, a) = \rho_2(x, a_x^2) \quad \text{and} \quad \sup_{b \in B_2} \rho_2(y, b) = \rho_2(y, b_y^2).
\]
\end{example}

\bibliographystyle{plain}
\bibliography{biblio2022.01}

\end{document}